\documentclass[11pt]{amsart}
\usepackage{paralist, amsmath, amsthm, amssymb, color, graphicx, hyperref, mathtools,mathrsfs}

\newtheorem{thm}{Theorem}
\numberwithin{thm}{section}
\newtheorem{theorem}[thm]{Theorem}
\newtheorem{proposition}[thm]{Proposition}

\newtheorem{lemma}[thm]{Lemma}

\theoremstyle{definition}
\newtheorem{definition}[thm]{Definition}

\theoremstyle{remark}

\newtheorem{remark}[thm]{Remark}
\newtheorem{conjecture}[thm]{Conjecture}

\newtheorem{example}[thm]{Example}

\newcommand{\bC}{\mathbb{C}}
\newcommand{\bN}{\mathbb{N}}

\newcommand{\bZ}{\mathbb{Z}}

\newcommand{\Qsym}{\ensuremath{\operatorname{QSym}}}

\newcommand{\suchthat}{\;|\;}

\newcommand{\asc}{\mathrm{asc}} 
\newcommand{\dsc}{\mathrm{des}} 

\newcommand{\mc}[1]{\mathcal{#1}} 
\newcommand{\sym}{\mathrm{Sym}} 

 \newcommand{\qint}[1]{(#1)_q} 
\newcommand{\qfact}[1]{(#1)_q!} 
\newcommand{\aseq}[1]{\mathrm{a}(#1)} 
\newcommand{\kly}{\mc{K}}

\newcommand{\acy}{\mathrm{Acy}}

\newcommand{\pathalg}{\mc{P}} 
\newcommand{\sfn}{\mathsf{n}} 
\newcommand{\sfe}{\mathsf{e}} 
\newcommand{\msf}[1]{\mathsf{#1}} 
\newcommand{\staircase}{\mathscr{S}} 
\newcommand{\zigzag}{\mathscr{Z}} 
\newcommand{\rectangular}{\mathscr{R}} 

\newcommand{\rect}{\square} 
\newcommand{\sourceseq}{\mathrm{ss}} 
\newcommand{\initial}{\mathrm{in}} 
\newcommand{\litspace}{\,}

\title[DU algebras and chromatic symmetric functions]{Down-up algebras and chromatic symmetric functions}
\author{Philippe Nadeau}
\address{Univ Lyon, Universit\'e Claude Bernard Lyon 1, CNRS UMR
5208, Institut Camille Jordan, 43 blvd. du 11 novembre 1918, F-69622 Villeurbanne cedex, France}
\email{\href{mailto:nadeau@math.univ-lyon1.fr}{nadeau@math.univ-lyon1.fr}}

\author{Vasu Tewari}
\address{Department of Mathematics, University of Hawaii at Manoa, Honolulu, HI 96822, USA}
\email{\href{mailto:vvtewari@math.hawaii.edu}{vvtewari@math.hawaii.edu}}
\thanks{P.~N is partially supported by the project ANR19-CE48-011-01 (COMBIN\'E). V.~T. acknowledges the support from Simons Collaboration Grant \#855592.}

\begin{document}

\begin{abstract}
We establish Guay-Paquet's unpublished linear relation between certain chromatic symmetric functions by relating his algebra on paths to the $q$-Klyachko algebra. 
The coefficients in this relation are $q$-hit polynomials, and they come up naturally in our setup as connected remixed Eulerian numbers, in contrast to the computational approach of Colmenarejo--Morales--Panova. 
As Guay-Paquet's algebra is a down-up algebra, we are able to harness algebraic results in the context of the latter and establish results of a combinatorial flavour. In particular we resolve a conjecture of Colmenarejo--Morales--Panova on chromatic symmetric functions. This concerns the abelian case of the Stanley--Stembridge conjecture, which we briefly survey.
\end{abstract}

\maketitle

\section{Introduction}
\label{sec:intro}

Let $G=(V,E)$ be a finite undirected graph with $V=[n]\coloneqq \{1,\dots,n\}$.
The chromatic quasisymmetric function $X_G$ introduced by Shareshian--Wachs~\cite{ShWa16} is a generalization of Stanley's chromatic symmetric function~\cite{St95}, which in turn is a generalization of Birkhoff's chromatic polynomial. Given the remarkable circle of ideas relating these functions to the cohomology of Hessenberg varieties~\cite[Section 10]{ShWa16} and the Stanley--Stembridge conjecture~\cite{StanleyStembridge}, these functions have garnered substantial attention in the last decade; see for instance~\cite{AbreuNigro, Ale,Alepan,Ath15,BC18,CMP,GP-hopf,HaradaPrecup,HuhNamYoo}. 

The Stanley--Stembridge conjecture states that $X_G$ is \emph{$e$-positive} when $G$ is the incomparability graph of a naturally-labeled unit interval order.
Such $G$ can be interpreted as Dyck paths $D$ and we refer to them as Dyck graphs, writing $X_D$ in place of $X_G$ when there is no scope for confusion.
While the aforementioned conjecture is still wide open, there are known partial cases, most notably the \emph{abelian case}~\cite{AbreuNigro,CMP,HaradaPrecup}.

Numerous lines of attack to this conjecture involve the \emph{modular law}~\cite{GP-modular, OS14}. This 
is a simple linear relation between certain $X_G$ which itself has been a subject of much investigation; see~\cite{PS22} for a deep geometric perspective.
Motivated by this law, Guay-Paquet~\cite{GP-notes} in unpublished work introduced the algebra $\pathalg$ as the \emph{noncommutative} algebra over $\bC(q)$ generated by $\sfn$ and $\sfe$ with the \emph{modular} relations:
\begin{align}
\label{eq:mod_relations_path_algebra}
(1+q)\msf{ene}&=q\msf{een}+\msf{nee}\\
\label{eq:mod_relations_path_algebra_1}
(1+q) \msf{nen}&=q\msf{enn}+\msf{nne}.
\end{align}
As we will see below, this algebra is in fact known as a down-up algebra. Working in $\pathalg$, Guay-Paquet~\cite[Theorem 1]{GP-notes} established a particularly elegant result which we now state. For undefined jargon in this context, we refer the reader to Sections~\ref{sec:path_alg_kly} and \ref{sec:abelian}.

\begin{theorem}[Guay-Paquet]
\label{th:intro_1}
Let $D=UVW$ be a Dyck path where $V$ is an abelian subpath with $m$ north steps \textup{(}denoted by $\msf{n}$\textup{)} and $n$ east steps \textup{(}denoted by $\msf{e}$ \textup{)}, with $m\geq  n$. In particular, $V$ may be identified with a partition $\lambda$ in an $m\times n$ box. Then 
\[
X_{UVW}=\sum_{0\leq k\leq n}\frac{H_k^{m,n}(\lambda)}{\qint{m}\qint{m-1}\cdots \qint{m-n+1}}\litspace X_{U\msf{e}^k\msf{n}^m\msf{e}^{n-k}W}.
\]
Here $H_k^{m,n}(\lambda)$ denotes the \emph{$q$-hit number} of Garsia--Remmel~\cite{GarsiaRemmel}, and $\qint{j}\coloneqq 1+q+\cdots+q^{j-1}$ is the $q$-analogue of $j$ for $j\geq 0$.
\end{theorem}

Informally put, abelian subpaths of Dyck paths may be replaced by special rectangular paths along with coefficients given by $q$-hit numbers.
Hence it suffices to study $X_G$ of the sort that arise on the right-hand side, thereby restricting attention to a much smaller family of Dyck graphs.

Recently, Colmenarejo--Morales--Panova~\cite{CMP} gave an independent proof of Theorem~\ref{th:intro_1} relying heavily on intricate rook-theoretic identities.  Yet another proof was given independently by Lee and Soh~\cite{LeeSoh22}.

\subsection{Discussion of results}
Our primary aim is to give a short algebraic proof of Theorem~\ref{th:intro_1} by relating the algebra $\pathalg$ to the \emph{$q$-Klyachko algebra} $\kly$.
This \emph{commutative} algebra is generated by $(u_i)_{i\in \bZ}$ subject to quadratic relations in~\eqref{eq:relation_q_Klyachko}, and its name reflects the fact that these relations are a deformation of Klyachko's presentation~\cite{Kly85} for the $S_n$-invariant part of the cohomology ring of the permutahedral variety. 
As the authors demonstrated in~\cite{NT21}, $\kly$ has intimate links with various subareas within algebraic combinatorics. This link to chromatic symmetric functions furthers our case.
$\kly$ possesses a basis $\mc{B}$ of square-free monomials and the statement in Theorem~\ref{th:intro_1} is equivalent to the expansion of the monomial $u_{1}^{c_1}\cdots u_k^{c_k}$ where $c_i>0$ for $i=1,\dots,k$ in terms of $\mc{B}$. 
The resulting coefficients are connected remixed Eulerian numbers~\cite{NT21,NTRemixed}.

The previous links unearth other interesting properties of $\pathalg$.
We briefly describe them postponing explicit statements.
It is the case that $\pathalg$ is a down-up algebra introduced by Benkart--Roby~\cite{BenRob} (see also \cite[Definition 4.14]{NCSFIV} which implies that $\pathalg$ is the $n=2$ case of the \emph{quantum pseudoplactic algebra}). 
As such it possesses a so-called PBW basis of staircase monomials $\staircase$, which was independently noticed by Guay-Paquet~\cite{GP-notes}.

It is then natural to inquire about the expansion of any monomial $w$ in the basis $\staircase$.
Colmenarejo--Morales--Panova conjectured~\cite[Conjecture 6.6]{CMP} that the resulting coefficients, up to sign, are Laurent polynomials with nonnegative integer coefficients.
We resolve this conjecture by giving a simple combinatorial rule in Section~\ref{sec:arbitrary_into_staircase}.

Yet another basis comes up as follows: one can identify the diagonal subalgebra $\pathalg^{\mathrm{diag}}\coloneqq \oplus_{i\geq 0}\pathalg_{i,i}$ with the polynomial subalgebra in $\kly$ generated by $u_0$ and $u_1$, these two generators corresponding to the words $\msf{en}$ and $\msf{ne}$ in $\pathalg^{\mathrm{diag}}$. 
 The monomials in $\msf{en}$ and $\msf{ne}$ thus form a linear basis of $\pathalg^{\mathrm{diag}}$, which can be extended to a third basis for the space $\pathalg_{i,j}$. We refer to this as the \emph{zigzag} basis; see Section~\ref{sub:zigzag} for the precise description. We give an explicit description for the expansion of any word in the alphabet $\{\msf{n},\msf{e}\}$ in this basis; see Theorem~\ref{th:zigzag_coefficient}.

In Section~\ref{sub:guay-paquet-rectangular}, we recall how the modular law implies that relations in $\pathalg$ translate to relations amongst chromatic symmetric functions. This leads immediately to the proof of Theorem~\ref{th:intro_1}, which is directly related to the abelian case of the Stanley--Stembridge conjecture. In Section~\ref{sub:abelian_compendium} we revisit that case, and attempt an understanding of how the two new formulae\textemdash{} those of Abreu--Nigro, Harada--Precup \textemdash{} can be related bijectively to the original work of Stanley \cite{St95}. 

\section{Graded down-up algebra}
\label{sec:pathalg}

\subsection{Some generalities}
\label{subsec:generalities}

A \emph{path} $P$ is any word $w\coloneqq w(P)$ in the alphabet $\{\msf{n},\msf{e}\}$. Pictorially we depict it by reading $w$ left to right and translating every instance of $\msf{n}$ (respectively $\msf{e}$) as a unit north step (respectively east step) beginning at the origin. We denote the number of $\msf{n}$'s (resp. $\msf{e}$'s) by $|w|_{\msf{n}}$ (resp. $|w|_{\msf{e}}$). We let $\lambda\coloneqq \lambda(P)$ be the partition (in English notation) naturally determined by $P$ in the top left corner of the $|w|_{\msf{n}} \times |w|_{\msf{e}}$ box.
Alternatively, given any $\lambda\subset m\times n$, we may reverse this association to get a path $P\coloneqq P(\lambda)$ starting from $(0,0)$ to $(n,m)$, which in turn determines a word $w(\lambda)$ in $\{\msf{n},\msf{e}\}$.
Thus we have the following objects naturally in bijection:
\begin{align*}
\{\lambda\subseteq m\times n\} \leftrightarrow \{\text{paths $P$ from $(0,0)$ to $(n,m)$}\} \leftrightarrow \{w\in \{\msf{n},\msf{e}\}^{m+n}\text{ with } |w|_{\msf{n}}=m\}.
\end{align*}
Thus we can, and will, interchangeably use $P$, $w$, or $\lambda$ if it is clear from context.

\subsection[Basic properties]{Basic properties of $\pathalg$}
\label{subsec:basic_prop}
Recall that $\pathalg$ is the $\bC(q)$-algebra generated by $\msf{n}$ and $\msf{e}$ subject to the modular relations~\eqref{eq:mod_relations_path_algebra} and \eqref{eq:mod_relations_path_algebra_1}.
It turns out that these modular relations imply that $\pathalg$ is an instance of a well-studied class of algebras called \emph{down-up algebras}. These were introduced by Benkart--Roby~\cite[Section 2]{BenRob} inspired by Stanley's work on differential posets~\cite{St88}.
In the notation of \emph{loc. cit.}, $\pathalg$ is the down-up algebra $A(1+q,-q,0)$. At $q=1$, this recovers the Weyl algebra. While the algebraic properties of down-up algebras have been thoroughly studied, that it encodes the modular law has hitherto not been noted, to the best of our knowledge.

By the \emph{PBW theorem} for down-up algebras~\cite[Theorem 3.1]{BenRob}, the set 
\[
\staircase=\{\msf{e}^a(\msf{ne})^b\msf{n}^c\suchthat a,b,c\in \bZ_{\geq 0}\}
\]
is a basis for $\pathalg$. 
We refer to its elements as \emph{staircase} monomials, and to $\staircase$ as the \emph{staircase basis}. In Section~\ref{sec:arbitrary_into_staircase}, we explain how to expand an arbitrary element of $\pathalg$ in this basis.

Observe that the modular relations preserve the number of $\msf{n}$'s and $\msf{e}$'s. We can use this information to endow $\pathalg$ 
with a $\bZ_{\geq 0}\times \bZ_{\geq 0}$-grading:
\begin{align}
\pathalg=\bigoplus_{m,n\in \bZ_{\geq 0}} \pathalg_{m,n},
\end{align}
where $\pathalg_{m,n}$ is spanned by words $w$ satisfying $|w|_{\msf{n}}=m $ and $|w|_{\msf{e}}=n$.

A particular graded piece that is relevant for us is $\pathalg^{\mathrm{diag}}$ defined by
\begin{align}
\pathalg^{\mathrm{diag}}=\bigoplus_{m\in\bZ_{\geq 0}} \pathalg_{m,m}.
\end{align}

\noindent\emph{Involution $\eta$.} Benkart--Roby~\cite[p. 329]{BenRob} consider the map $\eta$ swapping $\msf{n}$ and $\msf{e}$, and extend it to an algebra \emph{anti}automorphism of the free associative algebra generated by $\msf{n}$ and $\msf{e}$. Since the modular relations are preserved under this antiautomorphism, we get an involution $\eta$ on $\pathalg$ which is combinatorially natural.
The notion of transposing a partition $\lambda \subseteq m\times n$ to obtain $\lambda^t$ corresponds to reversing $\lambda(w)$ and switching $\msf{n}$'s for $\msf{e}$'s, and vice versa. This resulting word is precisely $\eta(\lambda(w))$.  
The map $\eta$ sends $\pathalg_{m,n}$ to $\pathalg_{n,m}$.
If $\mathscr{B}_{m,n}$ is any basis for $\pathalg_{m,n}$, then applying $\eta$ to each basis element gives a basis for $\mathscr{P}_{n,m}$.
This will allow us to work under the assumption that $m\leq n$ (or $n\leq m$) whenever convenient.
Notice also that the staircase basis $\staircase$ is stable under $\eta$.

\section[Basis expansions]{Basis expansions in the algebra $\pathalg$}
\label{sec:path_alg_kly}

We consider expansions of elements of $\pathalg$ in three different bases. The first one is the rectangular basis considered by Guay-Paquet~\cite{GP-notes} for which our main result is Theorem~\ref{th:GP}. Its proof makes use of the $q$-Klyachko algebra introduced by the authors~\cite{NT21,NTRemixed}. In Section~\ref{sec:abelian} we will obtain Theorem~\ref{th:intro_1} as a corollary.

We give two other expansions: first, in the staircase basis $\staircase$, thus proving a conjecture of Colmenarejo, Morales and Panova \cite{CMP}, and then in what we call the zigzag basis.

\subsection[Rectangular]{Expansion into the rectangular basis}
\label{sub:rectangular}

Given nonnegative integers $m\geq n$, define the set of \emph{rectangular monomials} as follows:
\begin{align}
\rectangular_{m,n}&=\{\msf{e}^k\msf{n}^{m}\msf{e}^{n-k}\suchthat 0\leq k\leq n\}.
\end{align}
For $m<n$, we obtain $\rectangular_{m,n}$ using $\eta$. 

Our aim in this section is to expand any word $w$ in terms of monomials in $\rectangular_{m,n}$. It is not clear that this can be done, but will become transparent in due course.
We first explain how $q$-hit polynomials show up in another context.

\subsubsection[Klyachko algebra]{The $q$-Klyachko algebra}
\label{subsec:qkly}

We give a brisk introduction to the $q$-Klyachko algebra covering the bare essentials and refer the reader to~\cite{NT21} for more details.
The \emph{$q$-Klyachko algebra} $\kly$ is the commutative, graded $\bC(q)$-algebra with generators $(u_i)_{i\in\bZ}$ and quadratic relations
\begin{align}
\label{eq:relation_q_Klyachko}
(1+q)u_i^2=qu_iu_{i-1}+u_iu_{i+1}
\end{align}
for all $i\in\bZ$. As the authors demonstrated in~\cite{NT21}, $\kly$ has intimate links with various subareas within algebraic combinatorics. 
The link to chromatic (quasi)symmetric functions in this article adds to these various connections.

If $c=(c_i)_{i\in\bZ}$ is a sequence of nonnegative integers with finite support,\footnote{The support of $c$ is the set of indices $i$ such that $c_i>0$.} let $u^c\coloneqq \prod_{i\in\bZ} u_i^{c_i}$. 
In the particular case where the entries of $c$ are $0$s or $1$s, we may identify $c$ with its support $I\subset \bZ$, and then let $u_I\coloneqq \prod_{i\in I}u_i$. We let $\mc{B}$ denote the entire collection of such squarefree monomials $u_I$.
By~\cite[Proposition 3.9]{NT21}, $\mc{B}$ is a basis for $\kly$.
We may thus decompose
\[u^c=\sum_Ip_c(I)u_I.\]
Let $m=|c|\coloneqq \sum_{i}c_i$. By homogeneity $p_c(I)=0$ unless $|I|=m$. 
We define \[
A_c(q)=(m)_q! \times p_c(\{1,\dots,m\}).
\]
It is zero if the support of $c$ is not contained in $\{1,\dots,m\}$; so we can consider $c=(c_1,\dots,c_m)$, and in this case $A_c(q)$ is a nonzero polynomial with nonnegative integer coefficients. 

These polynomials were introduced by the authors~\cite[Section 4.3]{NT21} under the name \emph{remixed Eulerian numbers}. 
Indeed they recover Postnikov's mixed Eulerian numbers~\cite[Section 16]{Pos09} at $q=1$; see~\cite{NTRemixed} for a deeper combinatorial study of these polynomials. 
We will only need them for some special $c$, as we describe next.

We say that $c=(c_1,\dots,c_m)$ with $|c|=m$ is \emph{connected} if its support is an interval $I$.
We can encode a family of $A_c(q)$ for connected $c$ via a generating function: For $\alpha=(\alpha_1,\ldots,\alpha_k)\vDash m$ a \emph{strong }composition, we have the identity~\cite[Proposition 5.6]{NT21}
\begin{equation}
\label{eq:gf_connected}
\sum_{j\geq 0}\prod_{i=1}^k\qint{j+i}^{\alpha_i}\litspace t^j=\frac{\sum_{i=0}^{m-k}A_{0^i\alpha 0^{m-k-i}}(q) \litspace t^i}{(t;q)_{m+1}}.
\end{equation}
Here $(t;q)_{m+1}=\prod_{1\leq i\leq m+1}(1-tq^{i-1})$ stands for the \emph{$q$-Pochhammer symbol}.

\begin{remark}
It is in fact the case that $\frac{A_{0^i\alpha 0^{r-k-i}}(q)}{\qfact{k}}$ is a polynomial with nonnegative integer coefficients; see proof of~\cite[Proposition 5.4]{NT21} for the general picture.
\end{remark}

We next record another generating function identity that is suspiciously similar to \eqref{eq:gf_connected}.

\subsubsection{Hit numbers and connected remixed Eulerians}
\label{subsec:hit_remixed}

Consider a partition $\lambda$  inside an $m\times m$ square. Following~\cite{GarsiaRemmel}, up to the $q$-exponent variation discussed in~\cite{CMP}, the \emph{$q$-hit numbers} $H_j^m(\lambda)$ can be defined by: 
\begin{align}
\label{eq:qhit_gf}
\sum_{k\geq 0}t^k\prod_{1\leq i\leq m}\qint{i-\lambda_{m+1-i}+k}=\frac{\sum_{j=0}^m H_j^m(\lambda)\litspace t^j}{(t;q)_{m+1}},
\end{align}
which ought to be compared to \eqref{eq:gf_connected}. For the sake of completeness we give a quick combinatorial description for the $q$-hit number $H_k^{m,n}(\lambda)$ where $\lambda\subseteq m\times n$. The $q$-hit numbers in \eqref{eq:qhit_gf} correspond to the case $m=n$.

Let $R(m,n,\lambda,k)$ denote the set of maximal nonattacking rook placements on an $m\times n$ board such that there are exactly $k$ rooks in the Ferrers board corresponding to $\lambda$.
Given $p\in R(m,n,\lambda,k)$ we let $\mathrm{stat}(p)$ denote the number of \emph{unattacked} cells in the $m\times n$ board. 
Unattacked cells are certain cells that do not contain rooks and are defined as follows.
A cell in $\lambda$ is unattacked if it does not lie below a rook, or to the right of a rook, or to the left of a rook outside $\lambda$.
A cell outside $\lambda$ is unattacked if it does not lie below a rook or to the right of a rook outside $\lambda$.
This given, we have
\begin{align*}
H_k^{m,n}(\lambda)=\sum_{p\in R(m,n,\lambda,k)}q^{\mathrm{stat}(p)}.
\end{align*}
If $m=n$, we write $H_k^{m}(\lambda)$.
It is straightforward to check that, assuming $m\geq n$, that
\begin{align}
\qfact{m-n}\litspace \times H_k^{m,n}(\lambda)=H_k^{m}(\lambda).
\end{align}
See Figure~\ref{fig:qhit} for a maximal nonattacking rook placement $p$ where $m=4$, $n=5$ and $\lambda=(3,3,1,0)$. The six unattacked cells tell us that $q^{\mathrm{stat}(p)}=q^6$, which is the contribution of $p$ to $H_{2}^{4,5}(\lambda)$.
\begin{figure}
\includegraphics[scale=0.75]{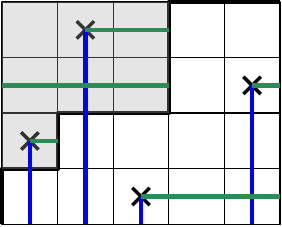}
\caption{Unattacked cells in a maximal nonattacking rook placement on a $4\times 5$ board}
\label{fig:qhit}
\end{figure}

\smallskip

To relate~\eqref{eq:gf_connected} and~\eqref{eq:qhit_gf}  we need some notation. Assume $m=n$ and $\lambda\subset m\times m$.
Define the \emph{area sequence} $\aseq{\lambda}\coloneqq (a_1,\dots,a_m)$ by setting \[
a_i=i-\lambda_{m+1-i}.
\] 
As $i$ goes from $1$ to $m$, the  $a_i$ go from $1-\lambda_m\leq 1$ to $m-\lambda_1\geq 0$ with `increments' in $\{1,0,-1,\dots\}$. 
It follows that the set of entries underlying $\aseq{\lambda}$ is an interval containing $0$ or $1$. 
Note further that the multisets underlying $\aseq{\lambda}$ and $\aseq{\lambda^t}$ are equal \textemdash{} as may be seen by a standard pairing of north and east steps at the same height for instance.

Now consider the monomial $u(\lambda)$ in $\kly$ defined as follows:
\begin{equation}
\label{eq:ulambda}
u(\lambda)\coloneqq \prod_{1\leq i\leq m} u_{a_i}.
\end{equation}
Clearly, $u(\lambda)$ depends solely on the multiset underlying $\aseq{\lambda}$.

\begin{example}
Consider $\lambda=(5,5,3,3,3,0)\subset 6\times 6$ as shown in Figure~\ref{fig:area_seq}. We have $\aseq{\lambda}=(1,-1,0,1,0,1)$ and $\aseq{\lambda^t}=(1,0,1,-1,0,1)$. Additionally, $u(\lambda)=u_{-1}u_0^2u_1^3$. 
\begin{figure}[!htbp]
\includegraphics[scale=0.4]{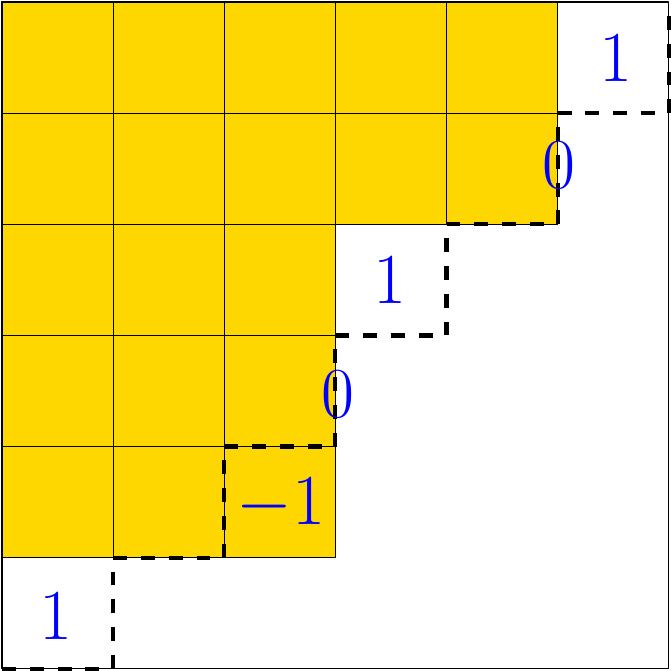}
\caption{$\lambda=(5,5,3,3,3,0)$ inside a $6\times 6$ board with $a(\lambda)=(1,-1,0,1,0,1)$.}
\label{fig:area_seq}
\end{figure}
\end{example}

It turns out that the coefficients when one expresses $u(\lambda)$ in the basis $\mc{B}$ are relevant to us.
The fact that $u(\lambda)$ has degree $m$, and that the set underlying $\aseq{\lambda}$ is an interval in $\bZ$ containing $0$ or $1$, 
implies an expansion in the basis $\mc{B}$ as follows:
\begin{equation}
\label{eq:uM_expansion}
 u(\lambda)=\sum_{k=0}^m c_k u_{[1,m]\downarrow k}.
\end{equation}
Here $[1,m]\downarrow k\coloneqq \{1-k,2-k,\dots,m-k\}$.
As established in~\cite[\S 4.2]{NTRemixed}, we have that
\begin{equation}
\label{eq:c_k_equals_q_hit}
 c_k=\frac{H_k^m(\lambda)}{(m)_q!}.
\end{equation}

The result next is essentially in~\cite{GP-notes} though not stated as such. The reader should compare this statement to Theorem~\ref{th:intro_1}: as we will see in Section~\ref{sub:guay-paquet-rectangular}, it will in fact imply it.
\begin{theorem}
\label{th:GP}
Fix nonnegative integers $m\geq n$. Let $\lambda\subset m\times n$, and 
consider the corresponding path $w=w(\lambda)$. Then in $\pathalg_{m,n}$ we have
\[
w(\lambda)=\sum_{k=0}^n\frac{H_{k}^{m,n}(\lambda)}{\qint{m}\qint{m-1}\cdots \qint{m-n+1}} \msf{e}^{k}\msf{n}^{m}\msf{e}^{n-k}.
\]
\end{theorem}
\begin{proof}
We first consider the case $m=n$. Consider the map $\psi:\pathalg_{m,m}\to \mathcal{K}$ sending
\begin{align}
w(\lambda) &\mapsto u(\lambda),
\end{align}
extended by linearity.

 We need to show that $\psi$ is well defined. To this end, we must verify that the result is unchanged when modular relations~\eqref{eq:mod_relations_path_algebra},\eqref{eq:mod_relations_path_algebra_1} are applied to $w$.

Applying the relation \eqref{eq:mod_relations_path_algebra} by changing $(1+q)\msf{ene}$ in $w$ to $q\litspace \msf{een}+\msf{nee}$ corresponds to changing an $a_i$ in the sequence $a(\lambda)$ to either $a_{i}+1$ or $a_i-1$. We can conclude with the Klyachko relation $(1+q)u_{a_i}^2=qu_{a_i}u_{a_{i}-1}+u_{a_i}u_{a_{i}+1}$, if we can find a $j\neq i$ such that  $a_j=a_i$. 

If $a_i\leq 0$ we are guaranteed that $a_{i+1}\leq a_i$. If it is equal then we are done. Otherwise $a_{i+1}<a_i$. Since $a_m\geq 0$ and increments in $a(\lambda)$ are bounded above by $1$, we must have a $j>i+1$ such that $a_j=a_i$. If $a_i \geq 1$ we apply this argument to $\mathrm{reverse}(w)$. Reversal preserves instances of $\msf{ene}$ and changes $a(\lambda)=(a_1,\dots,a_m)$ to $(1-a_m,\dots,1-a_1)$.
Thus an instance of $a_i\geq 1$ translates to $1-a_{m+1-i}\leq 0$ and we are back in the former setting.

 The case of the relation\eqref{eq:mod_relations_path_algebra_1}, namely $(1+q)\msf{nen}=q\litspace \msf{enn}+\msf{nne}$ can be dealt with similarly: it is simpler, since the occurrence of $\msf{nen}$ implies that we have the needed $u_{a_i}^2$ in the image already. Thus we have proved that $\psi$ is well defined.
\smallskip

Now note that $\pathalg_{m,m}$ has dimension $m+1$~\cite[Theorem 3.1]{BenRob}. Indeed the staircase monomials $\delta_i\coloneqq \msf{e}^{m-i}(\msf{ne})^i\msf{n}^{m-i}$ for $0\leq i\leq m$ give a basis.

Consider the $m+1$ \emph{rectangular} monomials $\rect_i=\msf{e}^{i}\msf{n}^m\msf{e}^{m-i}$.
 Since $\psi(\rect_i)=u_{[1,m]\downarrow i}$, we get that the $\rect_i$ are independent as their images are independent in $\kly$. 
We thus deduce  that the $\rect_i$  for $0\leq i\leq m$ give another basis of $\pathalg_{m,m}$.

It follows that the coefficients $c_{\lambda,i}$ in the expansion
\[
w=\sum_{0\leq i\leq m}c_{\lambda,i}\litspace\rect_i
\]
are those in the expansion
\[
\psi(w)=u(\lambda)=\sum_{0\leq i\leq m}c_{\lambda,i}\litspace u_{[1,m]\downarrow i}.
\]
Comparison with \eqref{eq:uM_expansion} implies the claim for $m=n$.
\smallskip

Assume now that $m>n$. We append $\msf{e}^{m-n}$ to $w$ at the end. This defines $w'=w(\lambda')$, where $\lambda'$ has the same shape as $\lambda$ but sits inside an $m\times m$ square. We can compute $a(\lambda')$ inside this square as before.
This forces the interval $[1,m-n]$ to be included in the set underlying $a(\lambda')$, so we can a priori restrict \eqref{eq:uM_expansion} to a smaller set of $n+1$ intervals:
\begin{equation}
\label{eq:uM_expansion_rewtricted}
 u(\lambda)=\sum_{k=0}^{n} \frac{H_k^m(\lambda)}{(m)_q!} \litspace u_{[1,m]\downarrow k}=\sum_{k=0}^{n}\frac{H_k^{m,n}(\lambda)}{\qint{m}\qint{m-1}\cdots \qint{m-n+1}}  \litspace u_{[1,m]\downarrow k}.
\end{equation}

Then the rest of the proof follows the same path as the square case.
We define $\psi$ as starting from $\pathalg_{m,n}$ by appending $\msf{e}^{m-n}$ to any element and then applying the map defined in the case $m=n$.
 The $n+1$ staircase monomials $\delta_{i,m,n}= \msf{e}^{n-i}(\msf{ne})^i\msf{n}^{m-i}$ are a basis $\pathalg_{m,n}$, so the $n+1$ rectangular monomials $\rect_{i,m,n}=\msf{e}^i\msf{n}^m\msf{e}^{n-i}$ also form one since their images $\psi(\rect_{i,m,n}\msf{e}^{m-n})$ are independent in $\kly$. We conclude that the coefficient of $\rect_{k,m,n}$ in the expansion of $w$ is given by $\frac{H_k^{m,n}(\lambda)}{\qint{m}\qint{m-1}\cdots \qint{m-n+1}}$.
\end{proof}
We consider an example next and return to the consequences of Theorem~\ref{th:GP} to chromatic symmetric functions in Section~\ref{sec:abelian}.
\begin{example}
Let $w=\msf{nennee}\in \pathalg_{3,2}$ and let $\lambda\coloneqq \lambda(w)=(1,1,0)$. Consider the six non-attacking rook placements on the $3\times 2$ board in Figure~\ref{fig:hit_rook}.
\begin{figure}
\includegraphics[scale=0.75]{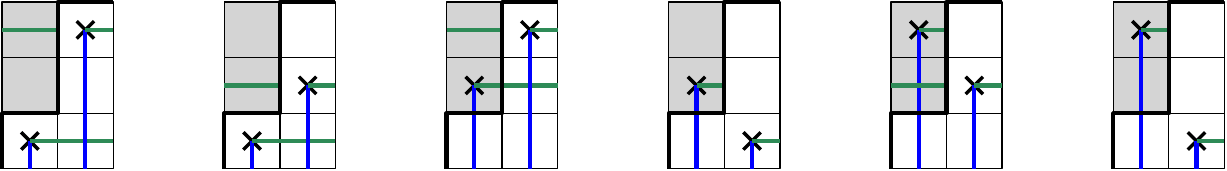}
\caption{Nonattacking rook placements on $3\times 2$ board}
\label{fig:hit_rook}
\end{figure}
The leftmost two rook placements contribute to $H_0^{3,2}(\lambda)$ and the remaining to $H_1^{3,2}(\lambda)$. We thus get
\begin{align*}
H_0^{3,2}(\lambda)&=q+q^2\\
H_1^{3,2}(\lambda)&=1+q+q^2+q^3.
\end{align*}
Theorem~\ref{th:GP} then says
\begin{align*}
\msf{nenne}=\frac{q+q^2}{\qint{3}\qint{2}}\litspace \msf{nnnee}+\frac{1+q+q^2+q^3}{\qint{3}\qint{2}}\litspace \msf{ennee}.
\end{align*}
\end{example}

\subsection{The staircase basis}
\label{sec:arbitrary_into_staircase}

Fix positive integers $m$ and $n$. Define $\staircase_{m,n}\coloneqq \staircase\cap \pathalg_{m,n}$. We know that $\staircase_{m,n}$ is a basis for $\pathalg_{m,n}$.
In this section we give an expansion for any monomial $w\in \pathalg_{m,n}$ in the basis $\staircase_{m,n}$.
Like before, we let $\delta_k\coloneqq\msf{e}^a(\msf{ne})^k\msf{n}^{b}$ where $a,b$ are such that $\delta_k\in \staircase_{m,n}$.

We begin by stating our claim. 
Let $m_w\geq 0$ be the largest integer such that the  path $P(\delta_{m_w})$ lies weakly below the path $P(w)$.

\begin{theorem}
\label{th:conj_cmp}
In  $\pathalg$, consider the basis expansion
\begin{equation}
\label{eq:staircase_expansion}
w=\sum_{0\leq k\leq m_w}(-1)^{m_w-k}c_{w,k}(q)\litspace\delta_k,
\end{equation}
Then $c_{w,k}\in \bZ_{\geq 0}[q]$ and vanishes unless $k\leq m_w$.
\end{theorem}

While there are in general many ways to employ the modular relation to express an arbitrary monomial $w$ in terms of staircase monomials, we are guided by
 the aim that $\msf{e}$'s and $\msf{n}'s$ move to the left and right respectively, and in doing so, force a string of $\msf{ne}$'s in between. 
At the same time, we want the signs to behave nicely in a predictable manner.
We will need solely the two relations:
\begin{align}
\label{eq:relation_eu}
\msf{n}^i\msf{e}&=\qint{i}\textcolor{blue}{\msf{ne}}\litspace \msf{n}^{i-1}-q\qint{i-1}\litspace \textcolor{blue}{\msf{en}}\litspace\msf{n}^{i-1}\\
\label{eq:relation_nu}
\msf{n}\msf{e}^i&=\qint{i}\msf{e}^{i-1}\litspace \textcolor{blue}{\msf{ne}}-q\qint{i-1}\litspace\msf{e}^{i-1}\litspace \textcolor{blue}{\msf{en}}.
\end{align}
These relations follow from the modular relations easily.
The second one follows from the first by applying the transpose $\eta$. 
Additionally, and crucially, observe that the coefficients involved are, up to a sign, polynomials in $\bZ_{\geq 0}[q]$.
\medskip

We state next our crucial definition that governs how the aforementioned relations apply in the course of our procedure.
\begin{definition}
Consider a factor $w'$ in $w$  where $w'=\msf{n}^i\msf{e}$ or $w'=\msf{ne}^i$ with $i\geq 2$ maximal. We say that $w'$ is \emph{critical} if $P(w)$ shares an edge with the path $P(\delta_{m_w})$ at one of the letters in $w'$.
,\end{definition}

Note that by definition of $m_w$, the letter in the critical factor that corresponds to $P(\delta_{m_w})$ is necessarily the starting $\msf{n}$ if  $w'=\msf{n}^i\msf{e}$, and the last $\msf{e}$ if $w'=\msf{ne}^i$.

\begin{lemma}
\label{lemma:critical}
Fix $w$ a word in $\{\msf{n},\msf{e}\}$ The following are equivalent. 
\begin{enumerate}
\item $w$ does not possess a critical factor.
\item $w$ corresponds to a staircase monomial.
\end{enumerate}
\end{lemma}
\begin{proof}
It is immediate that monomials in $\staircase_{m,n}$ do not contain critical factors. Hence assume $w\notin \staircase_{m,n}$ and consider the path $P(\delta_{m_w})$. It agrees, i.e. shares an edge, with $P(w)$ at various junctures. At one of the two extremes (or both) of any maximal factor of agreement, there must be a critical factor for $w$. At the right extreme, this will be a factor of the form $\msf{n}^i\msf{e}$. At the left extreme this will be the transposed version, i.e. $\msf{n}\msf{e}^i$. 
\end{proof}

\begin{figure}
\includegraphics[scale=0.7]{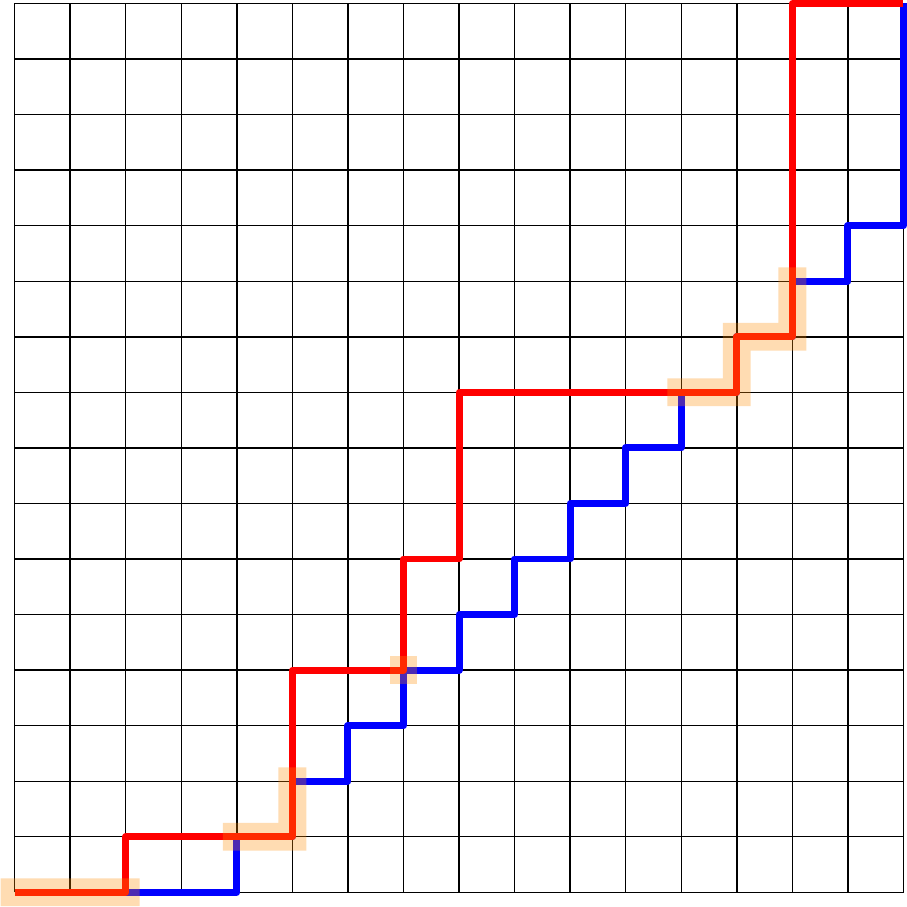}
\caption{A path $P(w)$ (in red) and the associated $P(\delta_{m_w})$ (in blue). The subpaths where the two touch are highlighted.}
\label{fig:critical}
\end{figure}

\noindent\textbf{The rewriting procedure:}
We now describe a rewriting procedure that takes as input any linear combination of words $C=\sum_wf_ww$. 
\begin{enumerate}
\item Pick $w$ such that $f_w\neq 0$ and $w$ possesses a critical factor. If no such $w$ exists, the procedure terminates and outputs $C$. 
\item Pick any critical factor $v$ in $w$. 
Modify $C$ by replacing the critical factor $v$ in $w$ according to the relations \eqref{eq:relation_eu},\eqref{eq:relation_nu} applied from left to right. Go back to the first step. 
\end{enumerate}

In the second step of the procedure, let $w_V,w_H$ be the two words that are obtained from a word $w$ after applying the relations~\eqref{eq:relation_eu},\eqref{eq:relation_nu}. Here $w_V$ comes with a positive weight $\qint{i}$, while $w_H$ comes with a negative weight $-q\qint{i-1}$.

Figure~\ref{fig:example_algo} shows an execution of  this algorithm for $w=\msf{nneeenne}$, representing naturally the rewriting procedure as a binary tree. We omitted the weights on the edges to keep the picture legible.
  
\begin{figure}
\includegraphics[scale=0.65]{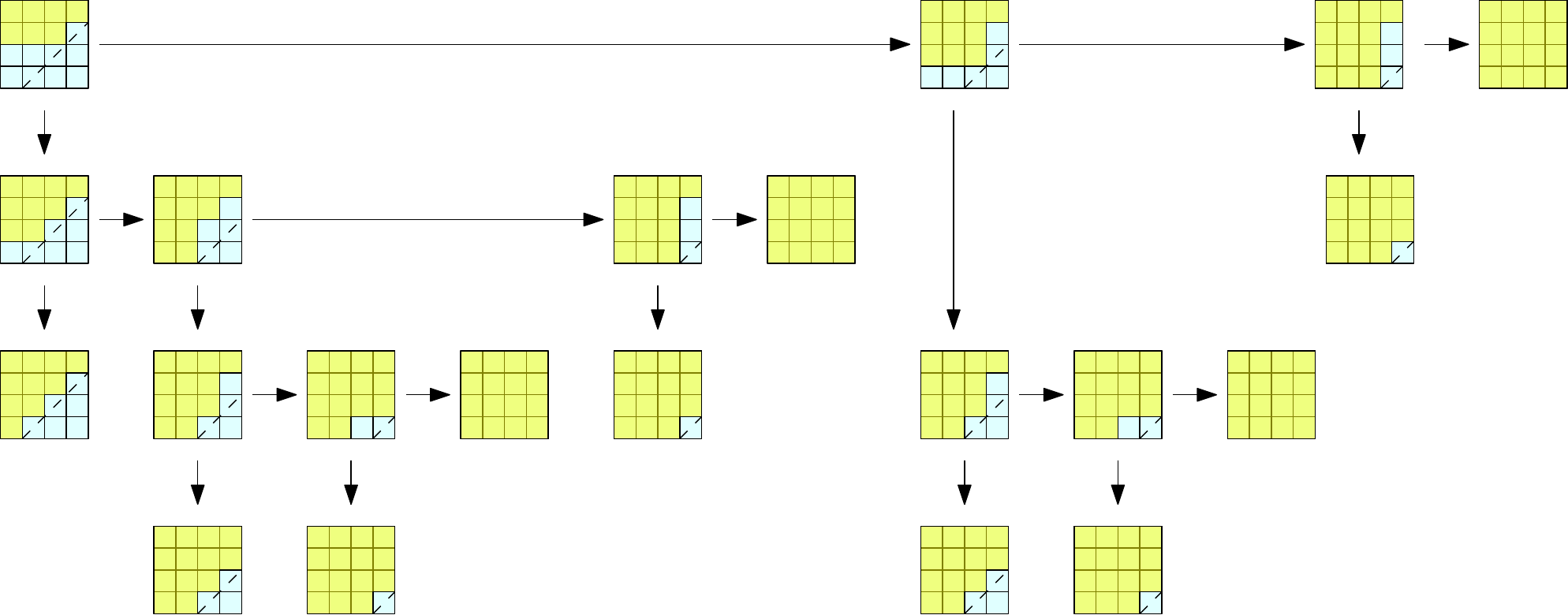}
\caption{Rewriting algorithm applied to $w=\msf{nneeenen}$. Horizontal (respectively vertical) arrows represent $w\to w_H$ (respectively $w\to w_V$).}
\label{fig:example_algo}
\end{figure}

\begin{proof}[Proof of Theorem~\ref{th:conj_cmp}]

First note that the rewriting procedure will necessarily end, as the shapes corresponding to the words are strictly increasing after each step of the procedure. It follows that the final output will be a linear combination of words with no critical factors, which represents the same element in $\pathalg$ as the starting linear combination since we only apply relations that are valid in $\pathalg$. By Lemma~\ref{lemma:critical}, this will indeed be the expansion into staircase monomials as desired.

Now a key remark is that for any word $w$, and any of its critical factors,we have  $m_{w_V}=m_w$ while $m_{w_H}=m_w-1$, where $w_H,w_V$ are defined above. 
Since $m_{\delta_k}=k$, any sequence of rewritings that goes from $w$ to $\delta_k$ will then necessarily involve $m_w-k$ sign switches, as $w_H$ comes with a negative weight while $w_V$ has a positive weight.

 It follows that the global sign of the coefficient of $\delta_k$ is $(-1)^{m_w-k}$, and thus that $c_{w,k}\in \bZ_{\geq 0}[q]$. It is also immediate from the procedure that $c_{w,k}=0$ if $k>m_w$.
\end{proof}

\begin{example}
Consider $w=\msf{nneeenen}$ as in Figure~\ref{fig:example_algo}.
There are exactly two paths from root to a leaf representing $\delta_2=\msf{eenenenn}$, both of which involve a single horizontal edge.
By considering the weights for each path we conclude that the coefficient of $\delta_2$ in $w$ is
\[
-q\qint{2}\qint{2}\qint{1}-q\qint{2}\qint{1}\qint{1}\qint{1}=-q(1+q)(2+q).
\]
\end{example}

\begin{remark}(Proof of ~\cite[Conjecture 6.6]{CMP})
Theorem~\ref{th:conj_cmp} implies easily~\cite[Conjecture 6.6]{CMP}.\footnote{Their conjecture is stated in terms of chromatic symmetric functions, but we explain in Section~\ref{sub:guay-paquet-rectangular} why this can be expressed in the algebra $\pathalg$.} 
The staircase basis in~\cite[Section 6]{CMP} corresponds to staircase paths in the top left corner. To expand into this basis, one needs to use the ``reverse'' rewriting rules, which are obtained from~\eqref{eq:relation_eu},\eqref{eq:relation_nu} by reversing the words and changing $q$ to $q^{-1}$:
\begin{align}
\label{eq:relation_eu_cmp}
\msf{e}^i\msf{n}&=q^{1-i}\qint{i}\textcolor{blue}{\msf{en}}\litspace \msf{e}^{i-1}-q^{1-i}\qint{i-1}\litspace \textcolor{blue}{\msf{ne}}\litspace\msf{e}^{i-1},\\
\label{eq:relation_nu_cmp}
\msf{e}\msf{n}^i&=q^{1-i}\qint{i}\msf{n}^{i-1}\litspace \textcolor{blue}{\msf{en}}-q^{1-i}\qint{i-1}\litspace\msf{n}^{i-1}\litspace \textcolor{blue}{\msf{ne}}.
\end{align}
This results in polynomials in $q^{-1}$ for the coefficients, instead of the polynomials in $q$ that we obtain in~\ref{th:conj_cmp}. 
\end{remark}

The proof in fact tells us a little bit more \textemdash{} we must have all $c_{w,k}\neq 0$ for $0\leq k\leq m_w$. Also, since the only way to hit the staircase monomial $\delta_{m_w}$ is by applying moves $w\to w_V$ at all stages, we get an explicit description for $c_{w,m_w}$ as a product of $q$-integers. For instance, for $w$ in Figure~\ref{fig:critical}, we have 
\[
c_{w,m_w}=\qint{3}\hspace{2mm} \qint{2}\qint{3}\qint{2}\qint{3}\qint{3}\qint{4}\qint{5} \hspace{2mm} \qint{6}\qint{5}.
\]

It is easy to give a characterization of this product in terms of $w$. More generally, it would be interesting to find a combinatorial interpretation for all the coefficients $c_{w,k}$.

\subsection{The zigzag basis}
\label{sub:zigzag}

For the purposes of this section, we set $s\coloneqq \msf{en}$ and $\msf{t}\coloneqq \msf{ne}$.

We return to the map $\psi$ defined in the proof of Theorem~\ref{th:GP}, except this  time we take its domain as $\pathalg^{\mathrm{diag}}$.
Then $\psi$ is an algebra homomorphism into $\kly$ since $u(\lambda\cdot\mu)=u(\lambda)u(\mu)$. As it sends a basis to independent vectors, it is  injective, and its image is the subalgebra of $\mathcal{K}$ with basis given by the $u_I$ with $I$ an interval containing $0$ or $1$. 
Equivalently, it is the subalgebra of $\kly$ generated by $u_0$ and $u_1$, which is free on the generators. In turn, this implies the following:

\begin{proposition}
\label{prop:Pdiag_polynomial}
$\pathalg^{\mathrm{diag}}$ is the (commutative) polynomial algebra $\bC(q)[\msf{s},\msf{t}]$.
\end{proposition}

\begin{remark}
Benkart--Roby~\cite[Proposition 3.5]{BenRob} establish that for a general down-up algebra $A(\alpha,\beta,\gamma)$, the subalgebra $A_0$ (i.e. the analogue of $\pathalg^{\mathrm{diag}}$) is always a commutative subalgebra. The proof in \emph{loc. cit.} is elementary albeit involved.\footnote{The reader should note that the grading employed in \cite{BenRob} is not our bigrading, but a weaker one that can be defined for any down-up algebra.}
Kirkman--Musson--Passman~\cite{KMP99} show that the subalgebra generated by $ud$ and $du$ in a general down-up algebra $A(\alpha,\beta,\gamma)$ over a field $K$ is a polynomial algebra in those two generators provided that $\beta\neq 0$. Recalling that $\pathalg$ is $A(1+q,-q,0)$, it is possible to apply their result in our context and obtain another proof of Proposition~\ref{prop:Pdiag_polynomial}.
\end{remark}

We are thus naturally led to the question of expanding monomial $w\in\pathalg_{m,m}$ in terms of $\msf{s}$ and $\msf{t}$. 
We consider a more general rectangular version.
Fix nonnegative integers $m\geq n$. Consider the set of \emph{zigzag} monomials defined as follows:
\begin{align}
\zigzag_{m,n}&=\{\msf{s}^a\msf{t}^{n-a}\msf{n}^{m-n}\suchthat 0\leq a\leq n\} 
\end{align}
For $m<n$, define $\zigzag_{m,n}$ by employing $\eta$. 
These zigzag monomials show up in~\cite[Section 2.1]{KMP99}.

Fix a word $w\in \pathalg_{m,n}$ with associated path and partition being $P$ and $\lambda$ respectively.
Define the sequence $b(\lambda)=(b_1,\dots,b_n)$ as follows:
\begin{align*}
b_i=\left\lbrace \begin{array}{cc}
m+1-i-\lambda'_i & \lambda_{m+1-i}<i\\
i-\lambda_{m+1-i} &\lambda_{m+1-i}\geq i. 
\end{array}\right.
\end{align*}
Informally, the sequence $b(\lambda)$ measures the distance  from the diagonal in the same vein as the area sequence $a(\lambda)$ from before. Figure~\ref{fig:b_sequence} gives an example where $m=11$ and $n=9$.
We either take the heights of the green shaded rectangle or the lengths of the red shaded rectangles. 
These capture the two cases that occur in the definition, and we get $b(\lambda)=(2,1,-1,0,4,5)$. 

\begin{figure}
\includegraphics[scale=0.7]{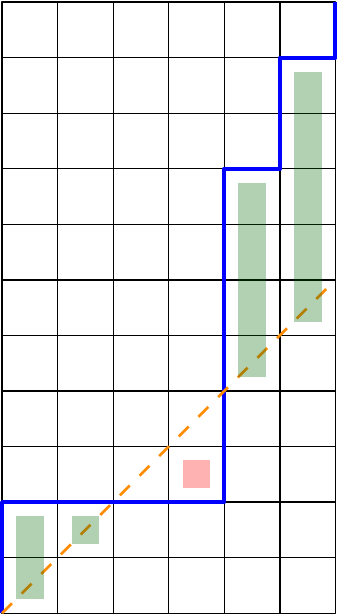}
\caption{}
\label{fig:b_sequence}
\end{figure}

For $i\in \bZ$ define $\mathrm{wt}_i\in \pathalg_{1,1}$ for $i\in \bZ$ as follows:
\begin{align*}
\mathrm{wt}_i=\left\lbrace \begin{array}{ll} 
\qint{i}\litspace \msf{t}-q\qint{i-1}\litspace \msf{s} & i\geq 1\\
q^{i}\left(\qint{1-i}\litspace \msf{s}-\qint{-i}\litspace \msf{t}\right) & i\leq 0.
\end{array}\right.
\end{align*}
This choice will become transparent during the course of the following proof. Note that $\mathrm{wt}_0=\msf{s}$ and $\mathrm{wt}_1=\msf{t}$.
\begin{proposition}
\label{prop:ne_en_factorization}
Fix a monomial $w\in \pathalg_{m,n}$ where $m\geq n$. 
We have
\begin{align*}
w=\mathrm{wt}_{b_1}\cdots \mathrm{wt}_{b_n}\cdot \msf{n}^{m-n}.
\end{align*}
\end{proposition}
\begin{proof}
If $n=0$, there is nothing to show as $w$ must necessarily equal $\msf{n}^{m-n}$. So we assume $n\geq 1$ and consider two cases.

Suppose $w=\msf{n}^i\msf{e}w'$ where $i\geq 1$. Then $i$ must necessarily equal $m-\lambda'_1$, which is $b_1$.
We thus have
\begin{align}
\label{eq:blah}
w=\msf{n}^{b_1}\msf{e}w'&=
\left((1-\qint{b_1})\msf{en}\cdot\msf{n}^{b_1-1}+\qint{b_1}\msf{ne}\cdot\msf{n}^{b_1-1}\right) w'=\mathrm{wt}_{b_1}\cdot \msf{n}^{b_1-1}w'.
\end{align}
Now $\msf{n}^{b_1-1}P'$ is a word representing a path in a smaller bounding box, and we can proceed by induction.

On the other hand, if $w=\msf{e}^i\msf{n}w'$ with $i\geq 1$, we must have $i=\lambda_{m}=1-b_1$.
Now mimicking the above argument we get
\begin{align}
\label{eq:blahblah}
w=\msf{e}^{1-b_1}\msf{n}w'&=
\left((1-(1-b_1)_{q^{-1}})\msf{ne}+(1-b_1)_{q^{-1}}\msf{en}\right)\msf{e}^{-b_1}w'\nonumber\\
&=q^{b_1}(\qint{1-b_1}\msf{en}-\qint{-b_1}\msf{ne})\msf{e}^{-b_1}w'=\mathrm{wt}_{b_1}\cdot \msf{e}^{-b_1}w'
\end{align}
Again $\msf{e}^{-b_1}P'$ is a word representing a path in a smaller bounding box and we may apply induction.
\end{proof}

We note that the $\mathrm{wt}_{b_i}$ all commute, so the product can be written in various ways.

\begin{example}
Referring to Figure~\ref{fig:b_sequence}, we have $w=\msf{n}^2\msf{e}^4\msf{n}^6\msf{e}\msf{n}^2\msf{en}$. Noting that $b(\lambda)=(2,1,-1,0,4,5)$, we get that
\[
w=(\qint{2}\litspace\msf{s}-q\litspace\msf{t})\cdot \msf{s}\cdot q^{-1}(\qint{2}\litspace\msf{t}-\msf{s})\cdot \msf{t} \cdot (\qint{4}\litspace\msf{s}-q\qint{3}\litspace\msf{t})\cdot (\qint{5}\litspace\msf{s}-q\qint{4}\litspace\msf{t})\cdot \msf{n}^{11-6}.
\]
\end{example}
\begin{theorem}
\label{th:zigzag_coefficient}
Consider the basis expansion in $\pathalg$
\begin{align*}
w=\sum_{0\leq r\leq n}c_{w,r}\litspace \msf{s}^r\msf{t}^{n-r}\msf{n}^{m-n}.
\end{align*}
Then $c_{w,r}$ is a \emph{globally signed} Laurent polynomial.
\end{theorem}
\begin{proof}
We extract the coefficient of  $\msf{s}^r\msf{t}^{n-r}$ in $\mathrm{wt}_{b_1}\cdots \mathrm{wt}_{b_n}$.  
Let $S\subset \binom{[n]}{r}$. Define $\mathrm{wt}_S$ as
\begin{align*}
\prod_{\substack{i\in S\\ b_i\geq 1}}\left(-q\qint{b_i-1}\right)\prod_{\substack{i\in S\\ b_i\leq 0}}\left(q^{b_i}\qint{1-b_i}\right)\prod_{\substack{i\notin S\\ b_i\geq 1}}\qint{b_i}\prod_{\substack{i\notin  S\\ b_i\leq 0}}\left(-q^{b_i}\qint{-b_i}\right)
\end{align*}
Now, $\mathrm{wt}_S$ is the coefficient that appears as one scans $\mathrm{wt}_{b_1}\cdots \mathrm{wt}_{b_n}$ left to right and picks up the coefficient of $\msf{s}$ if $i\in S$, and that of $\msf{t}$ if $i\notin S.$
It follows from Proposition~\ref{prop:ne_en_factorization} that
\begin{align*}
c_{w,r}=
\sum_{S\in \binom{[n]}{r}} \mathrm{wt}_S,
\end{align*}
For an $S$ to contribute to this expression, we must necessarily have all $i$ for which $b_i=0$ belong to $S$, and all $i$ for which $b_i=1$ belong to $[n]\setminus S$. If these constraints are not satisfied, then $\mathrm{wt}_S=0$.

Assuming these constraints are met, the sign of $\mathrm{wt}_S$ only depends on  $|S|$ and $w$. Indeed, the exponent of $-1$ is the number of $i\in S$ with $b_i> 1$ plus the  number of $i\notin S$ with $b_i< 0$. We leave it to the reader to verify that this quantity has the same parity as  $|S|$ plus the number of $i$ with $b_i\leq 0$. The claim follows.
\end{proof}

\begin{remark}
If $\lambda \subseteq m\times m$, then it is seen that $b(\lambda)$ is a rearrangement of the area sequence $a(\lambda)$ introduced in Section~\ref{subsec:hit_remixed}. So the form of Theorem~\ref{th:zigzag_coefficient} simplifies in the square case.
\end{remark}

\section{The abelian case of the Stanley--Stembridge conjecture}
\label{sec:abelian}

We relate here the algebra $\pathalg$ to  chromatic symmetric functions, following Guay-Paquet~\cite{GP-notes}. To keep our exposition brief, we refer the reader to~\cite[Chapter 7]{St99} for any undefined notions pertaining to the ring  $\Qsym$ of quasisymmetric functions, and its distinguished subring $\sym$ of symmetric functions. Given a strong composition $\alpha$, we let $M_{\alpha}$ and $F_{\alpha}$ denote the corresponding \emph{monomial}  and  \emph{fundamental} quasisymmetric functions respectively. 

\subsection{Chromatic quasisymmetric functions}
Consider a graph $G=([n],E)$. A \emph{coloring $\kappa$} of $G$ is an attribution of a \emph{color} in $\bZ_+=\{1,2,\ldots\}$ to each vertex of $G$; it is \emph{proper} if $\kappa(i)\neq \kappa(j)$ whenever $\{i,j\}\in E$. An {\em ascent \textup{(}respectively descent\textup{)}} of a coloring $\kappa$ is an edge $\{i<j\}\in E$ such that  $\kappa(i)<\kappa(j)$ (respectively $\kappa(i)>\kappa(j)$).
Denote the number of ascents (respectively descents) by $\asc(\kappa)$ (respectively $\dsc(\kappa)$).

The \emph{chromatic quasisymmetric function} of $G$~\cite{ShWa16} is the generating function of proper colorings weighted by ascents:
\begin{align}
\label{eq:def Xg}
X_G(x,q)=\sum_{\kappa:V\to \bZ_+ \text{proper}}q^{\asc(\kappa)}x_{\kappa(1)}x_{\kappa(2)}\dots x_{\kappa(n)}.
\end{align}
It is clearly in $\Qsym$, homogeneous of degree $n$. The chromatic symmetric function is $X_G(x,1)$ and was originally defined by Stanley~\cite{St95}.

 Letting $\rho$ be the linear involution $\rho$  on $\Qsym$ defined by sending $M_{\alpha_1,\ldots,\alpha_k}$ to $ M_{\alpha_k,\ldots,\alpha_1}$,
one has
\[\sum_{\kappa:V\to \bZ_+ \text{proper}}q^{\dsc(\kappa)}x_{\kappa(1)}x_{\kappa(2)}\dots x_{\kappa(n)}=q^{|E|}X_G(x,q^{-1})=\rho(X_G).
\]
Since $\rho$ leaves $\sym$ stable, we can use descents or ascents indifferently in the definition of $X_G$ when it happens to be symmetric, which is precisely the case we will be interested in.

As mentioned in the introduction, a particular class of graphs of interest to us are \emph{Dyck graphs}. 

\begin{definition}
A simple graph $G=([n],E)$ is a Dyck graph if for any $\{i<j\}\in E$, then $\{i'<j'\}\in E$ for all $i\leq i'<j'\leq j$.
\end{definition} 

Dyck graphs arise as incomparability graphs of natural unit interval orders; we will have no need for this description. A Dyck path $D$ uniquely determines a Dyck graph; Given all the  ways to index Dyck paths, we inherit various ways to index Dyck graphs, which we will employ.

\begin{proposition}[\cite{ShWa16}]
For $G$ a Dyck graph, $X_G(x,q)$ is a symmetric function.
\end{proposition}

\subsection{Guay-Paquet's rectangular formula}
\label{sub:guay-paquet-rectangular}

Let $G$ be a Dyck graph on $[n]$ corresponding to Dyck path $D$.
Let $I=\{i-a+1,\dots,i\}$, $J=\{j,j+1,\dots,j+{b-1}\}$ with $i<j$ be subsets of $[n]$ such that $(i-a+1,j-1)\in E$ and  $(i+1,j+b-1)\in E$. 
This forms an ``abelian rectangle'' $[i-a+1,i]\times [j,j+b-1]$. In terms of paths, this abelian rectangle corresponds to a certain ``abelian'' subpath of $D$ with $a$ north steps and $b$ east steps.

Figure~\ref{fig:abelian_subpath} depicts a Dyck path $D$. The labeled squares along the diagonal give the vertex set of the associated Dyck graph. Edges are given by squares below the path and above the diagonal squares. In this example, we have $I=\{2,3,4\}$ and $J=\{7,8\}$, and the resulting abelian rectangle $I\times J$ in gray. The subpath of $D$ in this shaded region gives the abelian subpath.

\begin{figure}
\includegraphics[scale=0.6]{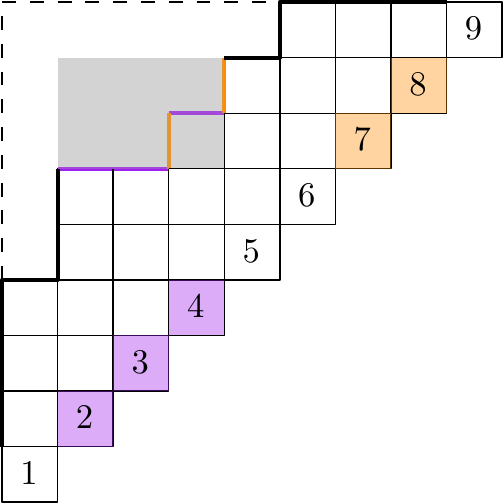}
\caption{A Dyck path with an abelian subpath in bold and the abelian rectangle highlighted.}
\label{fig:abelian_subpath}
\end{figure}
 
The modular law~\cite{GP-modular} says that if a subpath $\msf{ene}$ or $\msf{nen}$ is part of an abelian subpath, then
\begin{align}
\label{eq:XG_modular}
(1+q)X_{U\msf{ene}V}=qX_{U\msf{een}V}+X_{U\msf{nee}V},\\
(1+q)X_{U\msf{nen}V}=qX_{U\msf{enn}V}+X_{U\msf{nne}V}.
\end{align}

Fix $U$ and $W$, and consider the set of all Dyck paths $UvW$ where $v$ is the abelian subpath. Let us assume that $v$ has $a$ north steps and $b$ east steps, so that the abelian rectangle has dimensions $a\times b$.
We denote the $\bC(q)$-linear span of the chromatic symmetric functions $X_{UvW}$ by $\mc{X}_{a,b}$.

Consider the map on the $\bC(q)$-linear span of words with $a$ $\msf{n}$'s and $b$ $\msf{e}$'s, with image in $\mc{X}_{a,b}$, defined by
\begin{align}
v\mapsto X_{UvW}
\end{align}
and extended linearly. Comparing the relations \eqref{eq:mod_relations_path_algebra},\eqref{eq:mod_relations_path_algebra_1} of $\pathalg$ and the modular laws \eqref{eq:XG_modular}, we have in fact a map defined on $\pathalg_{a,b}$. We can thus apply this map to the relation in Theorem~\ref{th:GP}, and this gives precisely Theorem~\ref{th:intro_1}.

\subsection{The Stanley--Stembridge conjecture}
\label{sub:abelian_compendium}

The Stanley--Stembridge conjecture~\cite{StanleyStembridge,St95} asserts that, if $G$ is a Dyck graph then the chromatic symmetric function ${X_G}_{|_{q=1}}$ has a positive expansion in the $e_\lambda$ basis. Shareshian--Wachs~\cite{ShWa16} then extended it by conjecturing that the $e$-expansion of $X_G$ had coefficients that are polynomials in $q$ with nonnegative coefficients. Writing
\begin{equation}
\label{eq:expansion_elambda}
X_G=\sum_\lambda c_\lambda^G e_\lambda,
\end{equation}
the conjecture posits:

\begin{conjecture}
\label{conj:stanley_stembridge}
For any Dyck graph $G$ and any partition $\lambda$, the coefficient $c_\lambda^{G}$ is in $\bN[q]$.
\end{conjecture}

Recall that an \emph{acyclic orientation} $A$ of a graph $G$ is an orientation of its edges such that the resulting directed graph has no directed cycles. Assuming $V(G)=[n]$, an \emph{ascent} of $A$ is an edge $i\to j$ with $1\leq i<j\leq n$. If $G$ is nonempty, then $A$ has at least one \emph{source}, i.e. a vertex with no incoming edge. 

The following theorem was proved in~\cite[Theorem 5.3]{ShWa16}, the case $q=1$ being already known to Stanley~\cite[Theorem 3.3]{St95}.

\begin{theorem}
\label{th:sum_clambda}
For any Dyck graph $G$ and any $k\geq 1$, the sum of $c^G_\lambda$ over all partitions with $k$ parts is the number of acyclic orientations of $G$ with $k$ sources, counted with weight $q^{\#\text{ ascents of }A}$.
\end{theorem}

In particular the sum over all $\lambda$ of $c^G_\lambda$ is enumerated by the acyclic orientations of $G$. 

\subsection{The abelian case}

One case has been particularly studied and proved in different ways, called the \emph{abelian case}. 
In the language of the Section~\ref{sub:guay-paquet-rectangular}, this is when the Dyck graph $G$ on $n$ vertices has an associated abelian rectangle of maximal size $a\times b$ with $a+b=n$. In terms of Dyck paths, it means that the number of initial $\msf{n}$'s plus the number of final $\msf{e}$'s is $\geq n$; equivalently, the associated shape $\lambda=\lambda(G)$ satisfies $\lambda_1+\ell(\lambda)\leq n$. \smallskip 
 
We will now record and comment on two known $e$-expansions of $X_G$ when $G$ is abelian.

The \emph{source sequence} $\sourceseq(A)=(m_1,\ldots,m_k)$ of $A$ is defined recursively as follows: if $S_1$ is the set of sources of $A$, then $m_1=|S_1|$ and $(m_2,\ldots,m_k)$ is the source sequence of the acyclic orientation obtained by restricting $A$ to $G\setminus S_1$. 

Let $G$ be an abelian Dyck graph with $\lambda=\lambda(G)$. Let $(a_1,\ldots,a_n)$ be its ascent sequence. We also assume $\lambda_1\geq \ell=\ell(\lambda)$ without loss of generality. Since the vertices of $G$ can be partitioned in two cliques, acyclic orientations can have at most two sources. The expansion of $X_G$ thus only involves partitions with at most two parts.

\subsubsection{The formulas of Stanley and Harada and Precup}
\label{subsub:SHP}

Harada and Precup~\cite[Theorem 1.1]{HaradaPrecup} gave a proof of \ref{conj:stanley_stembridge}. They used the celebrated work of Brosnan and Chow~\cite{BC18} that showed the connection of $X_G$ for any Dyck graph $G$ with the study of Hessenberg varieties. Their result can be readily formulated as follows:
\begin{equation}
\label{eq:abelian_harada_precup}
X_G=|\acy^q_1(G)|\,e_{n}+\sum_{\{i<j\}\notin E}q^{a_i+a_j}X_{G\setminus\{i,j\}}^{+(1,1)},
\end{equation}
where $\acy^q_1(G)$ is the set of acyclic orientations of $G$ with one source, counted according to ascents; and  for a symmetric function $f=\sum c_\mu e_\mu$, then $f^{+(a,b,\dots)}\coloneqq\sum_{\mu}c_\mu e_{\mu_1+a,\mu_2+b,\dots}$. 
 
  Now by iterating the previous equation one obtains easily:
\begin{equation}
\label{eq:abelian_qstanley}
 X_G=\sum_A q^{\#\text{ ascents of }A}\,e_{n-\initial(A),\initial(A)},
\end{equation} 
where $\initial(A)$ is the length of the run of $2$'s at the beginning of $\sourceseq(A)$. The case $q=1$ is due to Stanley in his original paper~\cite[Theorem 3.4 and Corollary 3.6]{St95}. In fact, Stanley's proof can be extended to include $q$ and thus prove~\eqref{eq:abelian_qstanley}, which thus gives an independent proof of the result of Harada and Precup. 
\subsubsection{The formula of Abreu and Nigro}
A second proof was given by Abreu and Nigro~\cite[Theorem 1.3]{AbreuNigro}. Their result can be stated as follows:
\begin{equation}
\label{eq:abreu_nigro}
X_G=\sum_{j=0}^{\ell}q^j\qfact{j}\qint{n-2j}H^{n-j-1}_{j}(\lambda)\,e_{n-j,j}.
\end{equation}
Note that we slightly simplified their formula: the coefficient of $e_{n-\ell,\ell}$ in~\eqref{eq:abreu_nigro} is given in~\cite{AbreuNigro} as  $\qfact{\ell}H^{n-\ell}_{\ell}(\lambda)$.  

Let us explain why they coincide, which after simplifying by $\qfact{\ell}$ reduces to the identity
\begin{equation}
\label{eq:simple}
H^{n-\ell}_{\ell}(\lambda)=q^{\ell}\qint{n-2\ell} H^{n-\ell-1}_{\ell}(\lambda).
\end{equation}

\begin{proof}[Sketch of the proof of~\eqref{eq:simple}]
Write $N=n-\ell$. Fix a maximal rook configuration $C$ in $R(N-1,N-1,\lambda,\ell)$. Note that since $\ell=\ell(\lambda)$, all rooks in the top $\ell$ rows are necessarily inside $\lambda$, say in columns $J=\{j_1,\ldots,j_\ell\}$. One can extend $C$ to a configuration $C'$ in $R(N,N,\lambda,\ell)$ by inserting a rook in the bottom row in one of the $N-\ell$ columns $[N-1]\setminus J\cup\{N\}$. Tracking the new unattacked cells gives us the coefficient $q^{\ell}\qint{N-\ell}$: there are $\ell$ new unattacked cells in the top $\ell$ positions of the last column of $C'$, while $\qint{N-\ell}$ comes from the inversions created by the insertion in the last row.
\end{proof} 

\subsubsection{Comparison} It is certainly interesting to connect directly  \eqref{eq:abreu_nigro} to \eqref{eq:abelian_harada_precup},\eqref{eq:abelian_qstanley}. More precisely, equating  the two implies the following result

\begin{proposition}
Let $G$ be an abelian Dyck graph. Then the number of acyclic orientations $A$ with $\initial(A)=j$, with weight $q^{\#\text{ ascents of }A}$ is given by $q^j\qfact{j}\qint{n-2j}H^{n-j-1}_{j}(\lambda)$.
\end{proposition}

Let us sketch a direct bijective proof for $q=1$: Let $A$ be an acyclic orientation with $\initial(A)=j$. Let $S=(\{u_1<v_1\},\{u_2<v_2\},\ldots,\{u_j<v_j\})$ be the first $j$ sets in the source sequence decomposition of $A$. Denote by $V$ the set containing these $2j$ vertices. 
The orientation $A$ is then entirely characterized by $S$ together with an acyclic orientation $A_1$ of $G\setminus V$ that has a unique source by the definition of $\initial(A)$. Recall that the cells of $\lambda=\lambda(G)$ are in bijection with the non-edges of $G$. From this it follows that the vertices of $V$ can be represented by $j$ non-attacking roots in the shape $\lambda$, and they can be ordered in $j!$ ways.  Let $\lambda'\subset (n-2j)\times(n-2j)$ be the shape corresponding to $G\setminus V$: it is obtained by removing the columns and rows occupied by the rooks in $\lambda$. 
Now the number of acyclic orientations of $G\setminus V$ with a unique source  is given by $(n-2j)$ times the number $H^{n-2j-1}_0(\lambda')$, and this can be proved bijectively~\cite[\S 9.1]{Alepan}. 

Putting things together, we get a $1$-to-$j!(n-2j)$ map between acyclic orientations of $G$ with $\initial(A)=j$, and pairs of rook placements in $R(j,\ell,\lambda,j)\times R(n-2\ell-1,n-2\ell-1,\lambda',0)$ with $\lambda'$ as above. These two rook placements can be naturally combined to give a rook placement in $R(n-\ell-1,n-\ell-1,\lambda,j)$, which completes the bijective proof.

\section*{Acknowledgements}
We are extremely grateful to Mathieu Guay-Paquet for generously sharing his unpublished work. Thanks also to Ira Gessel for helpful correspondence in the context of $q$-hit numbers. Finally we would like to thank Laura Colmenarejo, Alejandro Morales, and Greta Panova for sharing an early version of their article.

\bibliographystyle{hplain}
\bibliography{Biblio_chromatic}

\end{document}